\documentclass[3p,times,procedia]{elsarticle}

\usepackage{ecrc}

\volume{00}

\firstpage{1}

\journalname{Neural Networks}

\runauth{}

\jid{nc}

\jnltitlelogo{Neural \\ Networks}

\CopyrightLine{2019}{Published by Elsevier Ltd.}

\usepackage[T1]{fontenc}
\usepackage{lmodern}
\usepackage{textcomp}
\usepackage{latexsym}
\usepackage{comment}
\usepackage{amsmath}
\usepackage{amssymb}
\usepackage{amsfonts}
\usepackage{mathrsfs}
\usepackage{amsthm}
\usepackage{listings}
\usepackage{algorithm}
\usepackage{algorithmic}
\newtheorem{thm}{Theorem}[section]
\newtheorem{defi}{Definition}[section]
\newtheorem{lem}{Lemma }[section]

\newtheorem{prop}{Proposition}[section]

\usepackage[figuresright]{rotating}

\begin{document}

\begin{frontmatter}



\dochead{}

\title{Variational Approximation Error in Non-negative Matrix Factorization}


\author[n,t]{Naoki HAYASHI}

\address[n]{Simulation \& Mining Division\\
 NTT DATA Mathematical Systems Inc. \\
1F Shinanomachi Rengakan, 35, Shinanomachi, Shinjuku-ku, Tokyo, 160--0016, Japan\\
hayashi@msi.co.jp}

\address[t]{Department of Mathematical and Computing Science\\
 Tokyo Institute of Technology \\
Mail-Box W8--42, 2--12--1, Oookayama, Meguro-ku, Tokyo, 152--8552, Japan\\
hayashi.n.ag@m.titech.ac.jp}

\begin{abstract}
Non-negative matrix factorization (NMF) is a knowledge discovery method that is used in many fields. Variational inference and Gibbs sampling methods for it are also well-known.
However, the variational approximation error has not been clarified yet, because NMF is not statistically regular and the prior distribution used in variational Bayesian~NMF (VBNMF) has zero or divergence points.
In this paper, using algebraic geometrical methods, we theoretically analyze the difference in negative log evidence (a.k.a. {\it free energy}) between VBNMF and Bayesian~NMF, i.e., the Kullback-Leibler divergence between the variational posterior and the true posterior.
We derive an upper bound for the learning coefficient (a.k.a. the {\it real log canonical threshold}) in Bayesian~NMF.
By using the upper bound, we find a lower bound for the approximation error, asymptotically.
The result quantitatively shows how well the VBNMF algorithm can approximate Bayesian~NMF; the lower bound depends on the hyperparameters and the true non-negative rank. A numerical experiment demonstrates the theoretical result.
\end{abstract}


\end{frontmatter}


\section{Introduction}
\subsection{Algorithms for Non-negative Matrix Factorization}
Non-negative matrix factorization (NMF) \cite{Paatero,Cemgil} has been applied to text mining \cite{Xu}, signal processing \cite{Lee}, bioinformatics \cite{Kim}, consumer analysis \cite{Kohjima}, and recommender systems \cite{Bobadilla2018recommender}.
NMF experiments discover knowledge and use it to make predictions about the future in the real world; however, the method suffers from many local minima and seldom reaches the global minimum.
Also, the results of numerical experiments have strongly depended on the initial values; a rigorous method has not yet been established.

To resolve these difficulties, Bayesian inference has been established for NMF \cite{Cemgil}.
It uses, as a numerical calculation of the Bayesian posterior distribution, the Gibbs sampling method, which is a kind of Markov chain Monte Carlo method (MCMC). Bayesian~NMF is more robust than the usual recursive methods of NMF, because it numerically expressed the posterior distribution; the parameters are subject to a probability distribution, and that makes it possible to determine the degree of fluctuation of the learning/inference result. As described later, the Bayesian method generally has higher estimation accuracy than maximum likelihood estimation and maximum posterior estimation if the model has a hierarchical structure or hidden variables, like in NMF.

On the other hand, a variational Bayesian algorithm (VB) for NMF has also been established \cite{Cemgil}, inspired by the mean-field approximation.
The variational Bayesian~NMF algorithm (VBNMF) is also more numerically stable than the usual recursive algorithms,
since VB approximates the Bayesian posterior distribution. Moreover, VBNMF is faster to compute than the usual Bayesian inference such as MCMC. However, its free energy (called the variational free energy) is larger than the Bayesian free energy, since VB raises the evidence lower bound (ELBO), but it is not the true model evidence. Note that the marginal likelihood is also called the model evidence, and the negative log ELBO is equal to the variational free energy.
From the above, it is important to clarify the approximation error of variational inference for not only theoretical reasons but also practical ones.

\subsection{Learning Theory of Bayesian and Variational Inference}
A statistical model is called regular if a function from a parameter set to a probability
density function set is one-to-one and if the likelihood function can be approximated by a Gaussian function.
It has been proven that if a statistical model is regular and if a true distribution is realizable by a statistical model, 
then the generalization error is asymptotically equal to $d/(2n)$,
where $d$, $n$, and the generalization error are the dimension of the parameter, the sample size, and
the expected Kullback-Leibler (KL) divergence of the true distribution and the estimated learning machine, respectively \cite{AkaikeAIC, SWatanabeBookMath}. 
Moreover, let $S_n$ be the empirical entropy; then, the negative log marginal likelihood or the free energy asymptotically behaves as $nS_n + (d/2) \log n + O_p(1)$ \cite{Schwarz1978BIC, SWatanabeBookMath}.
However, the statistical model used in NMF is not regular, because the map from a parameter to a probability
density function is not injective. As a result, its generalization error and free energy are still unknown. 

There are many non-regular statistical models in machine learning, for example,
reduced rank regressions,  normal mixture models, neural networks, hidden Markov models, and Boltzmann machines.
From a theoretical point of view, the free energy and the generalization error of a non-regular learning machine in Bayesian learning have been proved to be asymptotically equal to $nS_n + \lambda \log n$ and $\lambda/n$, respectively, where $\lambda$ is the real log canonical threshold  (RLCT) \cite{Watanabe1,Watanabe2}. Moreover, in non-regular cases, $\lambda < d/2$ holds, and $\lambda$ is also much less than the learning coefficients of maximum likelihood/posterior methods \cite{SWatanabeBookMath}.
The RLCTs for several learning machines have been studied. In fact, 
normal mixture models \cite{Yamazaki1}, 
reduced rank regressions \cite{Aoyagi1}, 
naive Bayesian networks \cite{Rusakov2005asymptotic}, 
hidden Markov models \cite{Yamazaki2}, 
Markov models \cite{Zwiernik2011asymptotic}, 
Gaussian latent tree and forest models \cite{Drton2017forest}, 
NMF \cite{nhayashi2, nhayashi5}, 
and latent Dirichlet allocation \cite{nhayashi7} 
have been studied by using resolution of singularities. A setting method for the temperature parameter of the replica Monte Carlo by using RLCTs has been studied \cite{Nagata2008asymptotic}. Moreover, a statistical model selection method {\it sBIC} using RLCTs has been proposed \cite{Drton}. 

On the other hand, for several statistical models, it was proved that the variational free energy asymptotically equals $nS_n + \lambda_{vb} \log n + O_p(1)$, where $\lambda_{vb}$ is a learning coefficient depending on the model;
examples include normal mixture models \cite{WatanabeK2006stochastic},
reduced rank regression \cite{Nakajima2007variational},
hidden Markov models \cite{HoshinoT2005hmmvb},
and NMF \cite{Kohjima2017phase}.
In general, the learning coefficient of VB may not be equal to the RLCT, but becomes an upper bound of it: $\lambda_{vb} \geqq \lambda$, since the variational free energy is larger than the free energy even if the sample size diverges to infinity.
Unfortunately, the variational generalization error is {\it not} asymptotically equal to $\lambda_{vb}/n$.
Moreover, variational inference seeks the mode of the true posterior,
and variational posterior distributions tend to concentrate in a fraction of the true posterior distributions.
One might say that a variational posterior distribution is equivalent to a posterior distribution; however, for these reasons, variational inference is only an approximation of Bayesian inference and is not equivalent.
As far as we know, there has been no direct theoretical comparison between them, except for \cite{Yamazaki2013comparing}.
The difference between VB and the local variational approximation, which is an approximation of VB, has been studied theoretically \cite{WatanabeK2012alternative}, but there have been few theoretical comparisons of Bayesian inference and VB. 


VBNMF has been devised \cite{Cemgil}, and the exact learning coefficient of VBNMF has been
derived \cite{Kohjima2017phase}. Nevertheless, the variational approximation error has not been found yet because the RLCT of NMF is still unknown. An upper bound of the RLCT of NMF has been proved if the prior distribution is a strictly and entirely positive bounded analytic function on the domain \cite{nhayashi2,nhayashi5}. Moreover, if the non-negative restriction is not assumed in the matrix factorization, the exact value of the RLCT is equal to that of the reduced rank regression models \cite{Aoyagi1}.
However, the RLCT is still unknown in the case of that the prior is a gamma distribution, which may be zero or infinite.

\subsection{Rest of the Paper}
In this paper, 
we theoretically give a lower bound for the KL divergence between the variational posterior and true posterior in NMF.
Let $\psi$ be the true posterior and $\psi_{vb}$ be the variational posterior. The KL divergence $\mathrm{KL}(\psi_{vb} \| \psi)$ is equal to $\overline{F}_n - F_n \approx (\lambda_{vb} - \lambda) \log n$. 
Hence, we theoretically derive a lower bound of the learning coefficient difference $\lambda_{vb} - \lambda$ in NMF, 
by which we can find a lower bound of the approximation error of VBNMF for Bayesian~NMF;
in this way, we can show how different the variational posterior is from the true posterior. 

The rest of this paper is organized as follows. 
In the second section, we briefly explain Bayesian inference and the variational Bayesian algorithm.
In the third section, we present the Main Theorems and sketches of their proofs.
In the fourth section, we discuss the results from theoretical application points of view and carry out numerical experiments aimed at verifying the numeric behavior of the Main Theorems.
The fifth section concludes this paper.
Rigorous proofs of the lemmas and theorems are in the appendices.

 \section{Inference Frameworks} 
 
 In this section, we explain the framework of Bayesian inference and the VB algorithm. 
 
 \subsection{Framework of Bayesian Inference} 
 
 First, we will briefly explain the general theory of Bayesian inference. 
 
Let $q(x)$ and $p(x|w)$ be probability density functions on a finite-dimensional real Euclidean space, where $w$ is a parameter;
$w$ is an element of a compact subset of a finite-dimensional real Euclidean space. 
In learning theory, $q(x)$ and $p(x|w)$ represent a {\it true distribution} and a {\it learning machine} with $w$ respectively. 
A probability density function on a set of parameters $\varphi(w)$ is called a {\it prior}. Usually, the prior has a parameter $\phi$ that is called a {\it hyperparameter}.
Let $X^n=(X_1,X_2,...,X_n)$ be a set of random variables that are independently subject to $q(x)$, where $n$ and 
$X^n$ are referred to as the {\it sample size} and {\it training data}, respectively. 
The {\it posterior} of $w$ is defined by
\[
\psi(w|X^n)=\frac{1}{Z_n}\varphi(w)\prod_{i=1}^n p(X_i|w) dw,
\]
where $Z_n$ is the normalizing constant that is determined by the condition $\int \psi(w|X^n)=1$. 
This constant $Z_n$ is called the {\it marginal likelihood, evidence, or partition function}, and it is also the probability density function of the training data: $Z_n = Z_n(X^n)$.
The Bayesian {\it predictive distribution} is defined by
\[
p(x|X^n)=\int p(x|w)\psi(w|X^n)dw. 
\]

The {\it negative log marginal likelihood} (or {\it free energy}) is defined by
\[
F_n = -\log Z_n = -\log \int \varphi(w)\prod_{i=1}^n p(X_i|w)dw.
\]
The generalization error $G_n$ is defined by the KL divergence from the true distribution $q(x)$ and
the predictive one $p(x|X^n)$:
\[
G_n=\int q(x)\log\frac{q(x)}{p(x|X^n)}dx.
\]
Note that $F_n$ and $G_n$ are functions of $X^n$; hence, they are also random variables. The expected value
of $G_n$ overall training data $\mathbb{E}[G_n]$ is called the expected generalization error, where the expectation $\mathbb{E}[\cdotp]$ is defined by
$$\mathbb{E}[\cdotp] = \int Q(x_1,\ldots,x_n) (\cdotp) dx_1 \ldots dx_n,$$
where
$$Q(X^n)=Q(X_1,\ldots,X_n)=\prod_{i=1}^n q(X_i).$$

Assume there exists a parameter $w_0$ that satisfies $q(x)=p(x|w_0)$. 
From singular learning theory \cite{Watanabe1, Watanabe2, SWatanabeBookMath}, it can be proven that 
\begin{align}
\label{WatanabeBayesF}
F_n &= nS_n + \lambda \log n -(m-1)\log\log n+ O_p(1), \\
\label{WatanabeBayesG}
\mathbb{E}[G_n] &=\frac{\lambda}{n}+o \left(\frac{1}{n} \right)
\end{align}
hold when $n$ tends to infinity,  even if the posterior distribution can {\it not} be approximated by
any normal distribution, where $S_n=-\frac{1}{n}\sum_{i=1}^n \log q(X_i)$ is the empirical entropy. The constant $\lambda$ is the {\it real log canonical threshold} (RLCT) which is
an important birational invariant in algebraic geometry. The constant $m$ is called the multiplicity and is also a birational invariant. From a mathematical point of view, RLCT is
characterized by the following property. We define a zeta function by
\[
\zeta(z)=\int K(w)^z\varphi(w)dw,
\]
where
\[
K(w)=\int q(x)\log\frac{q(x)}{p(x|w)}dx.
\]
This is holomorphic in $\mathrm{Re}(z)>0$ 
which can be analytically continued to a unique meromorphic function on the entire complex plane \cite{Atiyah1970resolution}. 
The poles of this extended function are all negative rational numbers. 
Let $(-\lambda)$ be the nearest pole to the origin; $\lambda$ is then equal to the RLCT. 
The multiplicity $m$ is denoted by the order of the nearest pole.
If the model is regular then $\lambda=d/2$ and $m=1$; however, NMF is not regular.

\subsection{Variational Bayesian Algorithm}

The variational Bayesian algorithm (VB), or variational inference, is an approximation method for Bayesian inference. VB is based on the mean-field approximation and is also called variational approximation. VB makes the cost of its numerical calculation is less than that of usual Bayesian inference.

Let the training data be $X^n=(X_1,\ldots,X_n)$ and the posterior be $\psi(w|X^n)=\psi(w^1,\ldots,w^d|X^n)$. In general, the posterior distribution cannot be found analytically; thus, we approximate the posterior $\psi(w|X^n)$ by the variational posterior defined by
$$\psi_{vb}(w|X^n):=\prod_{j=1}^d\psi^j_{vb}(w^j|X^n).$$


This approximation is performed by minimizing the KL divergence:
$$\min_{\psi_{vb}} \quad \mathrm{KL}(\psi_{vb}\| \psi)=\int \psi_{vb}(w|X^n)\log \frac{\psi_{vb}(w|X^n)}{\psi(w|X^n)}dw.$$
The problem of Bayesian inference is to find a numerical realization of the posterior and VB solves the problem by using the above optimization.
In practical cases, the parameters are often decomposed into parts, especially just two, and are assumed to be independent:
$$\psi(w|X^n) \approx \psi_{vb}(w|X^n)= \psi_{vb}^a(w_a|X^n) \psi_{vb}^b(w_b|X^n),$$
where $w=(w_a,w_b)$, $w_a=(w_{a_1},\dots,w_{a_k})$, $w_b=(w_{b_1},\dots,w_{b_{d-k}}).$

VB optimizes the above KL divergence by searching for $\psi_{vb}$, but the objective function
might be not able to be calculated analytically. This is because the marginal likelihood is contained in the KL divergence:
\begin{align*}
\mathrm{KL}(\psi_{vb}\| \psi)&=\int \psi_{vb}(w|X^n)\log \frac{\psi_{vb}(w|X^n)}{\psi(w|X^n)}dw \\
&=\int \psi_{vb}(w|X^n)\log \frac{\psi_{vb}(w|X^n)}{P(X^n|w)\varphi(w)/Z_n(X^n)}dw \\
&=\int \psi_{vb}(w|X^n)(\log\psi_{vb}(w|X^n) - \log P(X^n|w)\varphi(w))dw \\
&\quad + \int \psi_{vb}(w|X^n) \log Z_n(X^n)dw \\
&=\int \psi_{vb}(w|X^n)(\log\psi_{vb}(w|X^n) - \log P(X^n|w)\varphi(w))dw + \log Z_n(X^n),
\end{align*}
where $P(X^n|w)=\prod_{i=1}^n p(X_i|w)$.

The model $p(x|w)$ and the prior $\varphi(w)$ are designed manually, and the first term of the above KL divergence only contains known variables.
The minimization $\mathrm{KL}(\psi_{vb}\|\psi)$ problem is to minimize the following functional of $\psi_{vb}$:
$$\digamma(\psi_{vb}):=\int \psi_{vb}(w|X^n)(\log\psi_{vb}(w|X^n) - \log P(X^n|w)\varphi(w))dw.$$
If there exists $\hat{\psi}_{vb}$ such that $\hat{\psi}_{vb}=\psi$, then the KL divergence is equal to 0;
i.e., the functional is equal to the free energy: $\digamma(\hat{\psi}_{vb})=- \log Z_n(X^n) = F_n$.
Nevertheless, in general, there might not be any $\hat{\psi}_{vb}$ such that $\hat{\psi}_{vb}=\psi$.
The {\it variational free energy} $\overline{F}_n$ is defined by the minimum of $\digamma$:
$$\overline{F}_n = \min_{\psi_{vb}}\digamma(\psi_{vb}).$$
The inequality $\overline{F}_n \geqq F_n$ immediately follows from the definition.

For example, if one applies the approximation
$$\psi(w|X^n) \approx \psi_{vb}(w|X^n)= p_{vb}^a(w_a|X^n) p_{vb}^b(w_b|X^n),$$
then
\begin{align*}
\overline{F}_n &= \min_{\psi_{vb}=\psi^a_{vb}\psi^b_{vb}} \left\{ \int \psi_{vb}^a(w_a|X^n) \log \psi_{vb}^a(w_a | X^n) d w_a \right. \\
&\quad \left. + \int \psi_{vb}^b(w_b|X^n) \log \psi_{vb}^b(w_b|X^n) d w_b \right. \\
&\quad \left. + \iint \psi_{vb}^a(w_a|X^n) \psi_{vb}^b(w_b|X^n) \log  (P(X^n|w)\varphi(w)) d w_a d w_b \right\}.
\end{align*}

In addition, $\psi_{vb}^a$ and $\psi_{vb}^b$ satisfy the following {\it self-consistency condition}:
\begin{gather*}
\log \psi_{vb}^a(w_a|X^n){=}C_1{+}\int \! \psi_{vb}^b(w_b|X^n) \log \psi(w_a,w_b|X^n) d w_b, \\
\log \psi_{vb}^b(w_b|X^n){=}C_2{+}\int \! \psi_{vb}^a(w_a|X^n) \log \psi(w_a,w_b|X^n) d w_a,
\end{gather*}
where $C_1$ and $C_2$ are normalizing constants.

In this way, VB is an approximation of Bayesian inference and its accuracy can be expressed as $\overline{F}_n-F_n \geqq 0$;
the smaller the difference is, the more accurate the approximation is.
From a theoretical point of view, $\overline{F}_n$ has asymptotic behavior similar to that of $F_n$.
For example the following theorem has been proved for VBNMF \cite{Kohjima2017phase}.

\begin{thm}[Kohjima]\label{KohjimaVBNMF}
Let the elements of the data matrices $x_{ij}$ $(i=1,\ldots,M;$ $j=1,\ldots,N)$ be independently generated from
a Poisson distribution whose mean is equal to the $(i,j)$th element of $U_0V_0$, where the number of columns in $U_0$
$(= \mbox{the number of rows in } V_0)$ is $H_0$; $H_0$ is called the non-negative rank of $U_0V_0$ \cite{Cohen}. 

Let the likelihood model and the prior be the following Poisson and gamma distributions, respectively:
\begin{gather*}
p(X|U,V) = \prod_{i=1}^M \prod_{j=1}^N\frac{((UV)_{ij})^{x_{ij}}}{x_{ij}!}e^{-(UV)_{ij}}, \\
\varphi(U,V) = \mathrm{Gam}(U|\phi_U,\theta_U)\mathrm{Gam}(V|\phi_V,\theta_U),
\end{gather*}
where
\begin{gather*}
\mathrm{Gam}(U|\phi_U,\theta_U)=\prod_{i=1}^M\prod_{k=1}^H\frac{\theta_U^{\phi_U}}{\Gamma(\theta_U)}u_{ik}^{\phi_U-1}e^{-\theta_U u_{ik}}, \\
\mathrm{Gam}(V|\phi_V,\theta_V)=\prod_{k=1}^H\prod_{j=1}^N\frac{\theta_V^{\phi_V}}{\Gamma(\theta_V)}v_{ik}^{\phi_V-1}e^{-\theta_V v_{ik}},
\end{gather*}
are gamma distributions of the matrices $U$ and $V$, and $\phi_U, \theta_U, \phi_V, \theta_V > 0$ are hyperparameters,
where the size of $U$ and $V$ are $M \times H$ and $H \times N$, and $(UV)_{ij}$ is the $(i,j)$th entry of $UV$.

Assume the elements of $U_0V_0$ are strictly positive.
Then, the variational free energy $\overline{F}_n$ satisfies the following asymptotic equality:
$$\overline{F}_n = nS_n + \lambda_{vb} \log n + O_p(1) \ (n \rightarrow \infty),$$
where
$$\lambda_{vb} = \begin{cases}
(H-H_0)(M\phi_U+N\phi_V) + \frac{1}{2}H_0(M+N), & \mbox{if } M\phi_U + N\phi_V<\frac{M+N}{2} \\
\frac{1}{2}H(M+N), & \mbox{otherwise}.
\end{cases}$$
\end{thm}

In this paper, we mathematically show an upper bound $\overline{\lambda}$ for the RLCT $\lambda$ of the NMF in the case that is same as VBNMF; the model consists of Poisson distributions, and the prior consists of gamma distributions; the prior may have zero or divergence points in $K^{-1}(0)$.
By using the upper bound, we also derive a lower bound for the approximation error of VBNMF.

\section{Main Theorems and Proof}
In this section, we introduce the main results and prove them.
\subsection{Main Theorems}

In the following, $w=(U,V)$ is a parameter and $x=X$ is an observed random variable (matrix).

Let $\mathrm{M}(M,N,C)$ be a set of $M\times N$ matrices whose elements are in $C$, where $C$ is a subset of $\mathbb{R}$.
Let $K$ be a compact subset of $\mathbb{R}_{\geqq 0}=\{x \in \mathbb{R} | x \geqq 0\}$ and let 
$K_0$ be a compact subset $\mathbb{R}_{>0}=\{x \in \mathbb{R} | x > 0\}$.
Put $U \! \in \! \mathrm{M}(M,H,K)$, $V \! \in \! \mathrm{M}(H,N,K)$,
$U_0 \in \mathrm{M}(M,H_0,K_0)$ and $V_0 \in \mathrm{M}(H_0,N,K_0)$.
The non-negative rank \cite{Cohen} of $U_0V_0$, where $H \geqq H_0$ and $\{(x,y,a,b) \in K^2 \times K_0^2 |xy=ab \} \ne \emptyset$ are attained.

\begin{defi}[{\bf RLCT of NMF}]\label{def:NMFRLCT}
Assume that the largest pole of the function of one complex variable $z$, 
\[
\zeta(z)=\int_{\mathrm{M}(M,H,K) } dU \int_{\mathrm{M}(H,N,K) } dV
 \Bigl(\| UV-U_0V_0 \|^2\Bigr)^z \varphi(U,V)
\]
is equal to $(-\lambda)$. Then 
$\lambda$ is said to be the RLCT of NMF.
\end{defi}

We have already derived an upper bound of the RLCT of NMF
in the case that the prior $\varphi(U,V)$ is strictly positive and bounded \cite{nhayashi2,nhayashi5}.
However, to compare with VBNMF, we express the prior in terms of gamma distributions:
$$\varphi(U,V) = \mathrm{Gam}(U|\phi_U,\theta_U)\mathrm{Gam}(V|\phi_V,\theta_U).$$

In this paper, we prove the following theorems.

\begin{thm}\label{thm:main}
If the prior is the above gamma distributions, then the RLCT of NMF $\lambda$ satisfies the following inequality:
$$\lambda \leqq \frac{1}{2} \left[ (H-H_0)\min\{M\phi_U,N\phi_V \} +H_0 (M+N-1) \right].$$
The equality holds if $H=H_0=1$. 
\end{thm}

We prove this theorem in the next section. As two applications of it, we obtain upper bounds for the free energy and Bayesian generalization error of NMF in this case.

\begin{thm}\label{thm:bayes}
Under the same assumption as Theorem~\ref{KohjimaVBNMF},
the free energy $F_n$ and the expected generalization error $\mathbb{E}[G_n]$ satisfy the following inequality:
\begin{gather*}
F_n \leqq  nS_n + \frac{1}{2} \left[ (H-H_0)\min\{M\phi_U,N\phi_V \} +H_0 (M+N-1) \right] \log n + O_p(1), \\
\mathbb{E}[G_n] \leqq  \frac{1}{2n} \left[ (H-H_0)\min\{M\phi_U,N\phi_V \} +H_0 (M+N-1) \right]
+o\left(\frac{1}{n}\right).
\end{gather*}
\end{thm}

In Theorem~\ref{thm:bayes}, we study a case when a set of random matrices $X^n = (X_1,X_2,\ldots,X_n)$ 
are observed and the true decomposition $U_0$ and $V_0$ are statistically estimated. 
NMF has studied in the case when only one target matrix is decomposed;
however, in general, decomposition of a set of independent matrices should be studied
because the target matrices are often obtained daily, monthly, or in different places \cite{Kohjima}. In such cases, 
decomposition of a set of matrices amounts to making a statistical inference.
We will consider this situation, which is common to \cite{Kohjima2017phase} and Theorem~\ref{KohjimaVBNMF}.
A statistical model $p(X|U,V)$ which has parameters $(U,V)$ is used for the estimation.
The upper bounds of the free energy and generalization error in Bayesian inference are given by Theorem~\ref{thm:bayes}.

\begin{thm}\label{thm:BvsVB}
Let the variational free energy of VBNMF be $\overline{F}_n$. From the assumption of Theorem~\ref{KohjimaVBNMF}, the entries of $U_0V_0$ are not zero; thus $H_0>0$. Then, the following inequality holds:
$$\overline{F}_n -F_n \geqq \underline{\lambda} \log n + O_p(1),$$
where
$$\underline{\lambda} = \begin{cases}
\frac{1}{2}[(H-H_0)(M\phi_U+N\phi_V+\max\{M\phi_U,N\phi_V\})+H_0], & \mbox{if } M\phi_U + N\phi_V<\frac{M+N}{2} \\
\frac{1}{2}[(H-H_0)(M+N-\min\{M\phi_U,N\phi_V\})+H_0], & \mbox{otherwise}.
\end{cases}$$ 
\end{thm}

Theorem~\ref{thm:BvsVB} gives a lower bound for the KL divergence of the variational posterior and the true one (the difference between the variational free energy and the free energy).

\subsection{Lemmas and Sketches of Proofs for Main Theorems}
In order to prove Theorem~\ref{thm:main}, we use the following three lemmas proved in \ref{pf-lem}.

\begin{lem}
\label{lemH0}
Let $\lambda$ be the absolute of the maximum pole of
$$\zeta(z)=\iint dUdV\left(\| UV \|^2\right)^z \mathrm{Gam}(U|\phi_U,\theta_U)\mathrm{Gam}(V|\phi_V,\theta_V).$$
Then, $$\lambda = \frac{H\min\{ M\phi_U,N\phi_V \}}{2}$$
holds; this corresponds to equality in Theorem~\ref{thm:main} in the case $H_0=0$.
\end{lem}

\begin{lem}
\label{lemH1}
If $H_0=H=1$, equality in Theorem~\ref{thm:main} holds:
$$\lambda = \frac{M+N-1}{2}.$$
\end{lem}

\begin{lem}
\label{lemH}
If $H=H_0$, Theorem\ref{thm:main} holds:
$$\lambda \leqq \frac{H_0(M+N-1)}{2}.$$
\end{lem}

Let the entries of the matrices $(U,V)$ be
$$U\! = \!(u_1,\ldots,u_H) , u_{k}\!=\!(u_{ik})_{i=1}^{M},$$ 
$$V\! = \!(v_1,\ldots,v_H)^T , v_{k}\!=\!(v_{kj})_{j=1}^{N},$$
and
the ones of $(U_0,V_0)$ be
$$U_0\! =\! (u^0_1,\ldots,u^0_{H_0}) , u^0_{k}\!=\!(u^0_{ik})_{i=1}^{M},$$
$$V_0\! =\! (v^0_1,\ldots,v^0_{H_0})^T , v^0_{k}\!=\!(v^0_{kj})_{j=1}^{N}.$$

Now, let us sketch the proof of Theorem~\ref{thm:main} using the above lemmas.
The rigorous proof is given in \ref{pf-main}.

\begin{proof}[Sketch of Proof of Theorem~\ref{thm:main}]
Let $K$ and $L$ be non-negative analytic functions on a finite-dimensional Euclidean space.
A binary relation $\simeq$ is defined by
$$K(w) \simeq L(w) \Leftrightarrow_{def} \mbox{the RLCT of $K(w)$ is equal to the one of $L(w)$}.$$
Put
$$K_1(U,V)=\sum_{k=1}^{H_0}\left\| u_k (v_k)^T - u^0_k (v^0_k)^T\right\|^2$$
and
$$K_2(U,V)= 
\left\|
  \left(
    \begin{array}{ccc}
  u_{1(H_0 +1)} & \ldots & u_{1H} \\
  \vdots & \ddots & \vdots \\
  u_{M(H_0 +1)} & \ldots & u_{MH} \\
    \end{array}
  \right) 
  \left(
   \begin{array}{ccc}
  v_{(H_0 +1)1} & \ldots & v_{(H_0 +1)N} \\
  \vdots & \ddots & \vdots \\
  v_{H1} & \ldots & v_{HN} \\
    \end{array}
  \right)
\right\|^2.
$$

By developing $\| UV-U_0V_0 \|^2$, we have
\begin{eqnarray*}
  && \| UV-U_0V_0 \|^2  \\
  & \leqq & C \sum_{i=1}^M \sum_{j=1}^N 
\left( \sum_{k=1}^{H_0}  (u_{ik}v_{kj} - u^0_{ik}v^0_{kj})^2 + \sum_{k=H_0 +1}^H u_{ik}^2 v_{kj}^2 \right)  \ \rm{for} \ \exists C>0(\rm{const.}) \\
  &\simeq& K_1(U,V) + K_2(U,V).
\end{eqnarray*}

Because the prior $\varphi(U,V)\geqq 0$ and $K_1(U,V) >0$ in $\{(U,V) \mid \varphi(U,V)=0\}$, all we have to do is find the RLCT of
$$K_1(U,V) + K_2(U,V) \varphi(U,V)$$
to derive an upper bound.

Since $K_1(U,V)$ corresponds to the proof of Lemma \rm{\ref{lemH}} in the case that $H$ is replaced by $H_0$, the RLCT of $K_1(U,V)$ equals
$$\lambda_1 = H_0 \frac{M+N-1}{2}.$$
In contrast, $K_2(U,V) \varphi(U,V)$ corresponds to Lemma \rm{\ref{lemH0}} in the case that $H$ is replaced by $H-H_0$. That causes the RLCT of $K_2(U,V) \varphi(U,V)$ to be equal to
$$\lambda_2 = \frac{(H-H_0)\min\{M\phi_U,N\phi_V \}}{2}.$$
Considering the relation of the variables, the RLCT $\lambda$ satisfies the following inequality:
$$\lambda \leqq \frac{1}{2} \left[ (H-H_0)\min\{M\phi_U,N\phi_V\} +H_0 (M+N-1) \right].$$
\end{proof}

There is a proposition giving a sufficient condition for equality in Theorem~\ref{thm:main}.

\begin{prop}
Under the same assumption of Theorem~\ref{thm:main}, suppose 
\[
  f_{ij}^k = \begin{cases}
    u_{ik}v_{kj} - u^0_{ik}v^0_{kj} & (k \in \{1,\ldots,H_0 \}) \\
    u_{ik}v_{kj} & (k \in \{H_0+1,\ldots,H \})
  \end{cases}.
\]
If $f_{ij}^k \geqq 0$, i.e. , $ u_{ik}v_{kj} - u^0_{ik}v^0_{kj} \geqq 0 \quad (k \in \{1,\ldots,H_0 \})$, 
$$\lambda = \frac{1}{2} \left[ (H-H_0)\min\{M\phi_U,N\phi_V\} +H_0 (M+N-1) \right].$$
\end{prop}

\begin{proof}
Owing to $f_{ij}^k \geqq 0$, we have
$$\sum_{k=1}^H (f_{ij}^k)^2 \leqq \left( \sum_{k=1}^H f_{ij}^k \right)^2 \leqq H \sum_{k=1}^H (f_{ij}^k)^2.$$
Thus, 
$$\sum_{k=1}^H (f_{ij}^k)^2 \sim \left( \sum_{k=1}^H f_{ij}^k \right)^2.$$
Using the above relation, we get
\begin{eqnarray*}
  && \| UV-U_0V_0 \|^2  \\
  & \simeq & \sum_{i=1}^M \sum_{j=1}^N 
\left( \sum_{k=1}^{H}  (f_{ij}^k)^2 \right) \\
  &=& \sum_{i=1}^M \sum_{j=1}^N 
\left( \sum_{k=1}^{H_0}  (u_{ik}v_{kj} - u^0_{ik}v^0_{kj})^2 + \sum_{k=H_0 +1}^H u_{ik}^2 v_{kj}^2 \right) \\
  &\simeq& K_1(U,V) + K_2(U,V).
\end{eqnarray*}
$$\therefore \quad \lambda =  \frac{1}{2} \left[ (H-H_0)\min\{M\phi_U,N\phi_V\} +H_0 (M+N-1) \right].$$
\end{proof}

Next, Theorem~\ref{thm:bayes} and \ref{thm:BvsVB} are derived using Theorem~\ref{thm:main}.
Rigorous proofs are given in \ref{pf-main2}.

\begin{proof}[Sketch of Proof for Theorem~\ref{thm:bayes} and \ref{thm:BvsVB}]

Theorem~\ref{thm:bayes} is proved using the equalities (\ref{WatanabeBayesF}) and  (\ref{WatanabeBayesG}), and Theorem~\ref{thm:main}.
Theorem~\ref{thm:BvsVB} is derived as follows.
From Theorem~\ref{thm:bayes}, we have
$$F_n \leqq  nS_n + \overline{\lambda} \log n + O_p(1),$$
where
$$\overline{\lambda}=\frac{1}{2} \left[ (H-H_0)\min\{M\phi_U,N\phi_V \} +H_0 (M+N-1) \right].$$
Also, because of Theorem~\ref{KohjimaVBNMF}, 
$$\overline{F}_n = nS_n + \lambda_{vb} \log n + O_p(1)$$
holds, where
$$\lambda_{vb} = \begin{cases}
(H-H_0)(M\phi_U+N\phi_V) + \frac{1}{2}H_0(M+N), & \mbox{if } M\phi_U + N\phi_V<\frac{M+N}{2} \\
\frac{1}{2}H(M+N), & \mbox{otherwise}.
\end{cases}$$
We compute their difference
\begin{align*}
\overline{F}_n - F_n &= (\lambda_{vb}-\lambda)\log n + O_p(1) \\
&\geqq (\lambda_{vb} -\overline{\lambda}) \log n + O_p(1)
\end{align*}
both when $M\phi_U + N\phi_V<\frac{M+N}{2}$ and$M\phi_U + N\phi_V \geqq \frac{M+N}{2}$,
from which we obtain Theorem~\ref{thm:BvsVB}.
\end{proof}

\section{Discussion}

Here, we will discuss the results of this paper from three points of view.
After that, we will describe the numerical behavior of the theoretical result by conducting numerical experiments.

\subsection{Application to Model Selection}
First, we will explain an application of the Main Theorems. In this paper, we theoretically clarify the difference
between the variational free energy and the usual free energy in NMF.
From a practical point of view, the free energy $F_n$ can be calculated from the data; however, the entailed numerical integration is very hard and the sampling approximation, such an exchange Monte Carlo method, spends a long time finding $F_n$.
On the other hand, we can compute the variational one $\overline{F}_n$ more easily than $F_n$.
If the estimator of VBNMF is found, all we have to do is substitute it for the functional $\digamma$ whose minimum value is equal to $\overline{F}_n$.

It has not been clarified how much the variational free energy differs from the free energy; however, the Main Theorems give the lower bound. We can use the lower bound to approximate the free energy from the variational one.
Namely, when $\overline{F}_n$ is known, we have the approximation
$$F_n \approx \overline{F}_n - \underline{\lambda} \log n.$$
The usual VBNMF gives $F_n \approx\overline{F}_n$; here though, we can obtain a more accurate value\footnote{
Actually $\underline{\lambda}$ has the true non-negative rank $H_0$; however, in the same way as sBIC \cite{Drton}, 
we can avoid using the true knowledge by considering $H_0=0,\ldots,H$.}. In this way we should be able to
more accurately select the model in VBNMF by using $\overline{F}_n - \underline{\lambda} \log n$.

\subsection{Generalization Error}
Second, we describe the generalization error in NMF.
Theorem~\ref{thm:bayes} also gives an upper bound of the generalization error $G_n$ as well as the free energy $F_n$.
Generally speaking, the learning coefficients that control the asymptotic behavior of the $F_n$ and $G_n$ are the same
RLCTs \cite{SWatanabeBookMath}; hence, we can clarify both behaviors at once.
Since the situation in which the probability model $p(X|U,V)$ is a Poisson distribution and the prior $\varphi(U,V)$ is a gamma distribution is a case where the Gibbs sampling \cite{Cemgil} of NMF is performed, it can be regarded that not only $F_n$ but also $G_n$ are theoretically clarified when Gibbs sampling is applied.

By contrast, in Theorem~\ref{KohjimaVBNMF}, only the learning coefficient of the variational free energy $\overline{F}_n$ is determined.
This is because the learning coefficient of the variational generalization error is {\it not} equal to the one of $\overline{F}_n$.
Generally, in the case of VB, no zeta function is capable of uniformly handling $F_n$ and $G_n$, and the RLCT cannot obtain the learning coefficient\footnote{The learning coefficient of VB is not equal to RLCT.}.
For example, in VB of three-layered linear neural networks, the asymptotic behaviors are clarified not only with the variational free energy but also the variational generalization error \cite{Nakajima2007variational}, and their learning coefficients are different.
A linear neural network is also known as a reduced rank regression, a dimension reduction model, and the parameters are
equivalent to a matrix factorization model without a non-negative value constraint. In contrast in Bayesian inference in matrix factorization and NMF, the RLCT of matrix factorization is a lower bound for the RLCT of NMF, and it is known that the non-negative rank is dominant rather than the rank of the matrix in NMF, as described in \cite{nhayashi2, nhayashi5}.
Therefore, we cannot directly apply the results of linear neural networks to VBNMF.

In this way, theoretical generalization error in VB is rarely clarified, although that in Bayesian inference has been clarified with the free energy.
The Main Theorems show that Gibbs sampling is more reliable than VB, in the sense that it gives a theoretical guarantee not only about the free energy but also the generalization error. We can estimate the sample size
to achieve the needed inference performance and tune the hyperparameters.
Although various factors determine whether Gibbs sampling or VB is appropriate, our research can answer the question of whether or not the theoretical generalization error is clarified.

\subsection{Robustness on Probability Distributions}
Third, let us discuss the true distribution and the model of the data.
In this study, we consider the case in which the probability model $p(X|U,V)$ is a Poisson distribution and the prior $\varphi(U,V)$ is gamma distribution in the same way as in the derivation of the Gibbs sampling algorithm of NMF by Cemgil \cite{Cemgil}.
These assumptions are necessary for Gibbs sampling and the derivation of VB, but other models can be considered when using other MCMC methods. Is the main result applicable to these cases?

According to our prior research \cite{nhayashi2, nhayashi7},
several distributions satisfy the condition that the RLCT of NMF is equal to the absolute of the maximum pole of the following zeta function
\[
\zeta(z)=\int_{\mathrm{M}(M,H,K) } dU \int_{\mathrm{M}(H,N,K) } dV
 \Bigl(\| UV-U_0V_0 \|^2\Bigr)^z \varphi(U,V).
\]
Specifically, when the elements of the data matrix follow a normal distribution, a Poisson distribution, an exponential distribution,
or a Bernoulli distribution, 
the behavior of the free energy and the generalization error can be described using the same RLCT defined by the above zeta function.
In these previous studies, the prior distribution was limited to being positive and bounded, but when proving that the true distribution and the KL divergence of the probabilistic model have the same RLCT as the square norm error of the matrix, the prior distribution is arbitrary.

Therefore, if the prior distribution is a gamma distribution, the Main Theorems are valid not only when the probability model and the true distribution are Poisson distributions but also when they are normal distributions, exponential distributions, or Bernoulli distributions.
Indeed, if the hyperparameters are $\phi_U=\phi_V=1$, then the prior is strictly positive and bounded, and the upper bound equals
the result of \cite{nhayashi2}. Thus this study gives an extension to the case where the prior distribution is a gamma distribution of the main theorem of the previous work \cite{nhayashi2}.

\subsection{Experiment}

Here, we run numerical experiments to check the numeric behavior of the theoretical results.
Theorem~\ref{KohjimaVBNMF} gives the exact value of the learning coefficient $\lambda_{vb}$ of VBNMF and its validity was confirmed in Kohjima's previous research \cite{Kohjima2017phase}.
The core result is Theorem~\ref{thm:main}.
Therefore, we only has to run experiments for it; i.e., the RLCT $\lambda$ of Bayesian~NMF is calculated using Gibbs sampling. 



Let $\hat{\lambda}$ be the numerically calculated RLCT.
The widely applicable information criterion (WAIC) \cite{WatanabeAIC} is defined by the following random variable $W_n$:
$$W_n = T_n + V_n / n,$$
where $T_n$ is the empirical loss and $V_n$ is the functional variance:
\begin{align}
T_n &= -\frac{1}{n} \sum_{i=1}^n p^*(X_i) = -\frac{1}{n} \sum_{i=1}^n \mathbb{E}_w [p(X_i | w)], \\
V_n &= \sum_{i=1}^n \left[ \mathbb{E}_w [(\log p(X_i|w))^2] - \left\{ \mathbb{E}_w [ \log p(X_i|w)] \right\}^2 \right]
= \sum_{i=1}^n \mathbb{V}_w[\log p(X_i | w)].
\end{align}
Even if the posterior distribution cannot be approximated by any normal distribution (i.e., the model is singular),
the expected WAIC $\mathbb{E}[W_n]$ is asymptotically equal to the expected generalization loss $\mathbb{E}[G_n + S]$ \cite{WatanabeAIC};
$$\mathbb{E}[G_n + S] = \mathbb{E}[W_n] + o(1/n^2).$$
Moreover, the generalization error and $W_n - S_n$ have the same variance:
\begin{align}\label{var_RLCT}
G_n + W_n - S_n = 2\lambda /n + o_p(1/n).
\end{align}
Eq. (\ref{var_RLCT}) is useful for computing $\hat{\lambda}$ because the leading term $2\lambda /n$ is deterministic.
Nevertheless, the left hand side is probabilistic. This means that the needed number of simulation $D$ is smaller than that in the case using eq. (\ref{WatanabeBayesG}).

The method was as follows.
First, the training data $X^n$ was generated from the true distribution $q(X)$.
Second, the posterior distribution was calculated by using Gibbs sampling \cite{Cemgil}.
Third, $G_n$ and $W_n - S_n$ were computed by using the training data $X^n$ and the artificial test data $(X^*)^{n_T}$ generated from $q(X)$.
These three steps were repeated and each value of  $n(G_n + W_n - S_n)/2 was saved$.
After all repetitions have been completed,
$n(G_n + W_n - S_n)/2$ was averaged over the simulations. The average value was $\hat{\lambda}$.

The pseudo-code is listed in Algorithm \ref{exp_alg}, where $K$ is the sample size of the parameter subject to the posterior. We used the programing language named Julia 1.3.0 \cite{Bezanson2017julia} for this experiment.

\begin{algorithm}
\caption{How to Compute $\hat{\lambda}$}
\label{exp_alg}
\begin{algorithmic}
\REQUIRE $\phi = (\phi_U, \theta_U, \phi_V, \theta_V) >0$: the hyperparameters, \\
$U_0$: the true parameter matrix whose size is $(M, H_0)$,\\
$V_0$: the true parameter matrix whose size is $(H_0, N)$,\\
$\mathrm{GS}$: the Gibbs sampling function whose return value consists of the samples from the posterior.
\ENSURE The numerical computed RLCT $\hat{\lambda}$.
\STATE Allocate an array $\Lambda[D]$.
\FOR{$d=1$ to $D$}
\STATE Generate $X^n \sim q(X) = p(X|w_0)$, where $w_0 = (U_0, V_0)$.
\STATE Allocate arrays $\mathcal{U}[M, H, K]$ and $\mathcal{V}[H, N, K]$.
\STATE Get $\mathcal{U}, \mathcal{V} \leftarrow \mathrm{GS}(X^n, \phi)$.
\STATE Generate $(X^*)^{n_T} \sim q(X)$.
\STATE Calculate $G_n \approx \frac{1}{n_T} \sum_{t=1}^{n_T} \log \frac{q(X^*_t)}{\mathbb{E}_w [p(X^*_t | w)]}$,
$S_n = -\frac{1}{n} \sum_{i=1}^n \log q(X_i)$,
\STATE and $W_n \approx -\frac{1}{n} \sum_{i=1}^n \mathbb{E}_w [p(X_i | w)] + \frac{1}{n} \sum_{i=1}^n \mathbb{V}_w[\log p(X_i | w)]$,
\STATE $\quad$ where $\mathbb{E}_w[ f(w) ] \approx \frac{1}{K} \sum_{k=1}^K f(w_k)$ and $w_k = (\mathcal{U}[:,:,k], \mathcal{V}[:,:,k])$.
\STATE Save $\Lambda[d] \leftarrow n(G_n + W_n - S_n)/2$.
\ENDFOR
\STATE Calculate $\hat{\lambda} = \frac{1}{D} \sum_{d=1}^D \Lambda[d]$.
\end{algorithmic}
\end{algorithm}


We set 
$M=N=4$, $H=2$, $H_0=1$, and $n_T = 100n$.
To examine the behavior given different sample sizes, we set $n=500, 1000$ and $K=1000, 2000$, respectively.
To decrease the probabilistic effect of eq.~(\ref{var_RLCT}), we conducted the simulations twenty times, i.e., $D=20$.

The hyperparameters were set to $\theta_U = \theta_V =1$ and  $(\phi_U, \phi_V) = (0.25, 0.25), (0.5, 0.5), (1, 1), (2, 2)$.
We chose four pairs of $(\phi_U, \phi_V)$ in view of a theoretical point: the critical line of the learning coefficient in VBNMF.
Under the condition $M=N$, $\phi_U + \phi_V = 1$ is the phase transition line (see Theorem~\ref{KohjimaVBNMF}). 
Each point on the straight line $\phi_U = \phi_V$ is characterized as follows. $(0.25, 0.25)$ is before the phase transition line, and $(0.5, 0.5)$ is a phase transition point. $(1,1)$ is a case where the prior distribution $\varphi(U,V)$ is strictly positive and bounded, which was treated in the previous research \cite{nhayashi2, nhayashi5}. $(2,2)$ is beyond the critical line.

In the Gibbs sampling, we had to conduct a burn-in to decrease the effect of the initial values and thin the samples in order to break the correlations.
The sample size for the burn-in was $20000$, while the sample size for the thinning was $20$; thus, the sample sizes of the parameter was $20000 + 20K = 40000$ and $60000$ $(K=1000$ and $2000)$ and we used the $(20000 + 20k)$-th sample as the entry of $\mathcal{U}[:, :, k]$ and $\mathcal{V}[:, :, k]$ for $k=1$ to $K$.


\begin{table}[h]\centering
\caption{Numerically Calculated and Theoretical Values of the Learning Coefficients}
  \begin{tabular}{|c|c|c|c|} \hline
   \multicolumn{2}{|l|}{Hyperparam. / Coeff. }
      & \multicolumn{1}{c|}{$n=500$} & \multicolumn{1}{c|}{$n=1000$} \\ \hline
      & $\lambda_{vb}$ & 6 & 6 \\ 
    $\phi_U{=}\phi_V{=}0.25$ &  $\overline{\lambda}$ & 4 & 4 \\ 
    $\theta_U{=}\theta_V{=}1$ & $\hat{\lambda}$ & $3.74 \pm 0.0412$ & $3.73 \pm 0.0508$ \\ 
      & $\overline{\lambda}-\hat{\lambda}$ & $0.260 \pm 0.0412$ & $0.268 \pm 0.0508$ \\ \hline
      & $\lambda_{vb}$ & 8 & 8 \\ 
    $\phi_U{=}\phi_V{=}0.5$ &  $\overline{\lambda}$ &  9/2 & 9/2 \\ 
    $\theta_U{=}\theta_V{=}1$ &  $\hat{\lambda}$ & $4.05 \pm 0.0706$ & $4.15 \pm 0.0842$ \\ 
      & $\overline{\lambda}-\hat{\lambda}$ & $0.450 \pm 0.0706$ & $0.346 \pm 0.0842$ \\ \hline
      & $\lambda_{vb}$ &  8 & 8 \\ 
    $\phi_U{=}\phi_V{=}1$ &  $\overline{\lambda}$ & 11/2 & 11/2 \\ 
    $\theta_U{=}\theta_V{=}1$ & $\hat{\lambda}$ & $4.52 \pm 0.0492$ & $4.54 \pm 0.0628$ \\ 
      & $\overline{\lambda}-\hat{\lambda}$ & $0.976 \pm 0.0492$ & $0.965 \pm 0.0628$ \\ \hline
      &   $\lambda_{vb}$ &  8 & 8 \\ 
    $\phi_U{=}\phi_V{=}2$ &  $\overline{\lambda}$ & 15/2 & 15/2 \\ 
    $\theta_U{=}\theta_V{=}1$ & $\hat{\lambda}$ & $4.76 \pm 0.0454$ & $4.79 \pm 0.0477$ \\ 
      & $\overline{\lambda}-\hat{\lambda}$ & $2.74 \pm 0.0454$ & $2.71 \pm 0.0477$ \\ \hline
  \end{tabular}\label{exp-result}
\end{table}

\newpage

The experimental results are shown in Table~\ref{exp-result}.
The symbol $\overline{\lambda}$ denotes the theoretical upper bound of the RLCT $\lambda$ in Theorem~\ref{thm:main}.
There are columns for each sample size, and each row contains the hyperparameter (Hyperparam.) and the learning coefficient (Coeff.). The experimental values have three significant digits.


As shown in Table~\ref{exp-result}, all numerically calculated values are smaller than the theoretical upper bound,
thus, the Theorem~\ref{thm:main} is consistent with the experimental result.
Since $\lambda_N$ is larger in the case $\phi_U = \phi_V = 2$ than in the case $\phi_U=\phi_V=1$,
it is conjectured that the larger $\phi_U$ and $\phi_V$ are, the larger $\lambda$ will be,
while $\lambda_{vb}$ saturates to $H(M+N)/2=8$ (owing to Theorem~\ref{KohjimaVBNMF}) and the upper bound of $\lambda$ looks less tight than that in the other case.
In the case $\phi_U=\phi_V=0.5$, when the hyperparameter is on the phase transition line, $\lambda_N$ fluctuates more than in the other case.
We can see that learning is unstable on the critical line.
Hence, the hyperparameters should be set to avoid the neighborhood of the phase transition line.

\section{Conclusion}
We gave an upper bound of the real log canonical threshold for non-negative matrix factorization whose priors are gamma distributions and described theoretical applications to Bayesian and variational inference. Owing to the Main Theorems, the variational approximation error, i.e., the difference between the variational free energy and the free energy, can be quantitatively evaluated and the difference depends on the true non-negative rank and the hyperparameters of the gamma distributions. The numerical results are consistent with the theoretical results and suggest the exact values and the stability of learning. Future tasks include conducting large-scale experiments and clarifying the exact value of the real log canonical threshold.



\appendix
\section{Proofs of the Lemmas}\label{pf-lem}

Let $K(U,V)$ be the KL divergence between the true distribution $q(X)$ and the model $p(X|U,V)$
and $\Phi(U,V)=\|UV-U_0V_0\|^2$.
In the same way as \cite{nhayashi2}, $K(U,V) \simeq \Phi(U,V)$ follows.
Thus we consider the zero points of $\Phi(U,V)$ and $\varphi(U,V)$ to calculate the RLCT.

If $\Phi^{-1}(0) \ne \emptyset$, $\varphi^{-1}(0) \ne \emptyset$, and $\Phi^{-1}(0) \cap \varphi^{-1}(0)=\emptyset$,
then $\varphi(U,V)$ does not affect the maximum pole of the zeta function.
The prior $\varphi(U,V)$ is a gamma distribution, hence, 
$$\varphi^{-1}(0)=\left(\bigcup_{i=1}^M \bigcup_{k=1}^H \{(U,V)\mid u_{ik}=0\}\right) \cup \left( \bigcup_{k=1}^H \bigcup_{j=1}^N \{(U,V)\mid v_{kj}=0\}\right).$$
In the case of Lemma \ref{lemH0}, the set $\{ (U,V) \mid \|UV\|^2=0\}$ has intersections with $\varphi^{-1}(0)$.
However, in the case of Lemma \ref{lemH1}, the matrices $(U,V)\in \Phi^{-1}(0)$ cannot have zero elements.
Therefore, Lemma \ref{lemH1} can be proved in the same way as \cite{nhayashi2},
as can Lemma \ref{lemH}.

Here, we prove Lemma \ref{lemH0}.
\begin{proof}[Proof of Lemma \ref{lemH0}]
We will consider the simultaneous resolution $\Phi(U,V)=0$ and $\varphi(U,V)=0$,
since $\{(U,V) \mid \|UV\|^2=0\}$ has intersections with $\varphi^{-1}(0)$.

The zeta function is equal to
$$\zeta(z)=\iint dUdV\left(\| UV \|^2\right)^z \mathrm{Gam}(U|\phi_U,\theta_U)\mathrm{Gam}(V|\phi_V,\theta_V).$$

Since the elements of the matrices are non-negative, the inequality
$$\sum_{k=1}^H u_{ik}^2 v_{kj}^2 \leqq \Biggl( \sum_{k=1}^H u_{ik} v_{kj} \Biggr)^2 \leqq H \sum_{k=1}^H u_{ik}^2 v_{kj}^2$$
holds. Thus, we have
$$\sum_{i=1}^M \sum_{j=1}^N \sum_{k=1}^H u_{ik}^2 v_{kj}^2 \leqq \| UV \|^2 \leqq H \sum_{i=1}^M \sum_{j=1}^N \sum_{k=1}^H u_{ik}^2 v_{kj}^2.$$

As an RLCT is not changed by any constant coefficient
and it is order isomorphic; $\exists (c_1, c_2) \in \mathbb{R}^2 \ {\rm s.t.} \ c_1 F \leqq G \leqq c_2 F \Rightarrow F \simeq G$,
all we have to do is to calculate the RLCT of 
$$\sum_{i=1}^M \sum_{j=1}^N \sum_{k=1}^H u_{ik}^2 v_{kj}^2 
=\sum_{k=1}^H \Biggl(\sum_{i=1}^M u_{ik}^2 \Biggr) \Biggl( \sum_{j=1}^N v_{kj}^2 \Biggr).$$

The RLCT $\lambda$ becomes a sum of RLCTs about $k$. For each $k$, 
we blow-up the variables $(u_{ik})$, $(v_{kj})$ each other.
Simultaneously, we perform blow-ups of $\prod_{k=1}^H (\prod_{i=1}^M u_{ik}^{\phi_U-1} \prod_{j=1}^N v_{kj}^{\phi_V-1})$ and the determinant $\det J$ of their Jacobi matrix $J$; $\det J$ is called the Jacobian.

Let $A$ be either $U$ or $V$ and $\mathfrak{b}_{rc}[A]: (U,V) \mapsto (U', V')$ be a blow-up such that
the $(i,k)$th-entry of $U'$ and the $(k,j)$th-entry of $V'$ are equal to $U'_{ik}$ and $V'_{kj}$:
$$\begin{cases}
	U'_{ik} = \begin{cases}
		u_{ik} & (i=r \ {\rm and} \ k=c) \ {\rm or} \ k \ne c \\
		u_{rc}u_{ik} & i\ne r \ {\rm and} \ k=c \\
	\end{cases}, \quad V'_{kj}=v_{kj} & A=U, \\
	U'_{ik} = u_{ik}, \quad V'_{kj} = \begin{cases}
		v_{kj} & k=r \ {\rm and} \ j=c \\
		v_{rc}v_{kj} & {\rm otherwise}
	\end{cases} & A=V
\end{cases}.$$

For example, let $M=N=H=2$ and apply $\mathfrak{b}_{11}[U]$ to
$$\|UV\|^2 \simeq \sum_{k=1}^H \Biggl(\sum_{i=1}^M u_{ik}^2 \Biggr) \Biggl( \sum_{j=1}^N v_{kj}^2 \Biggr)
\ \mbox{and} \
\prod_{k=1}^2 \Biggl(\prod_{i=1}^2 u_{ik}^{\phi_U-1} \prod_{j=1}^2 v_{kj}^{\phi_V-1}\Biggr)
.$$
Thus, we have
$$U' =\left(
\begin{matrix}
	u_{11} & u_{12} \\
	u_{11}u_{21} & u_{22}
\end{matrix}
\right), \quad V'=\left(
\begin{matrix}
	v_{11} & v_{12} \\
	v_{21} & v_{22}
\end{matrix}
\right)$$
and
\begin{gather*}
\sum_{k=1}^2 \Biggl(\sum_{i=1}^2 u_{ik}^2 \Biggr) \Biggl( \sum_{j=1}^2 v_{kj}^2 \Biggr) \mapsto
u_{11}^2(1+u_{21}^2)(v_{11}^2 + v_{12}^2) + (u_{12}^2 + u_{22}^2)(v_{21}^2 + v_{22}^2), \\
\prod_{k=1}^2 \Biggl(\prod_{i=1}^2 u_{ik}^{\phi_U-1} \prod_{j=1}^2 v_{kj}^{\phi_V-1}\Biggr) \mapsto
u_{11}^{2(\phi_U-1)}u_{21}^{\phi_U-1}v_{11}^{\phi_V-1}v_{12}^{\phi_V-1}u_{12}^{\phi_U-1}u_{22}^{\phi_U-1}v_{21}^{\phi_V-1}v_{22}^{\phi_V-1}.
\end{gather*}
The Jacobi matrix is an $8 \times 8$ matrix as follows:
$$J=\frac{\partial (U',V')}{\partial(U,V)}
=\left(
\begin{matrix}
1 & 0 & u_{21} & 0 & 0 & 0 & 0 & 0 \\
0 & 1 & 0 & 0 & 0 & 0 & 0 & 0 \\
0 & 0 & u_{11} & 0 & 0 & 0 & 0 & 0 \\
0 & 0 & 0 & 1 & 0 & 0 & 0 & 0 \\
0 & 0 & 0 & 0 & 1 & 0 & 0 & 0 \\
0 & 0 & 0 & 0 & 0 & 1 & 0 & 0 \\
0 & 0 & 0 & 0 & 0 & 0 & 1 & 0 \\
0 & 0 & 0 & 0 & 0 & 0 & 0 & 1
\end{matrix}
\right)$$
and its determinant is $\det J=(u_{11})^{1} = u_{11}$. 

We choose $k \in \{1,\ldots,H\}$ arbitrarily and fix it. In the general case, because of the symmetry of the variables, only the blow-ups
$\mathfrak{b}_{1k}[U]$ and $\mathfrak{b}_{k1}[V]$ should be treated. Hence, we have
\begin{gather*}
\Biggl(\sum_{i=1}^M u_{ik}^2 \Biggr) \Biggl( \sum_{j=1}^N v_{kj}^2 \Biggr) \mapsto
u_{1k}^2v_{k1}^2\left(1+\sum_{i=2}^M u_{ik}^2\right)\left(1+\sum_{j=2}^N v_{kj}^2\right), \\
\prod_{i=1}^M u_{ik}^{\phi_U-1} \prod_{j=1}^N v_{kj}^{\phi_V-1} \mapsto
u_{1k}^{M\phi_U-M}v_{k1}^{N\phi_V-N}\prod_{i=2}^M u_{ik}^{\phi_U-1} \prod_{j=2}^N v_{kj}^{\phi_V-1}.
\end{gather*}
The Jacobian $(\det J)_{k}$ is equal to
$(\det J)_{k} = u_{1k}^{M-1} v_{k1}^{N-1}$.
The term $\left(1+\sum_{i=2}^M u_{ik}^2\right)\left(1+\sum_{j=2}^N v_{kj}^2\right)$ is strictly positive;
thus, we should consider the maximum pole of the following meromorphic function:
\begin{align*}
\tilde{\zeta}(z) &= \iint dUdV (u_{1k}^2v_{k1}^2)^z
\left(u_{1k}^{M\phi_U-M}v_{k1}^{N\phi_V-N}\prod_{i=2}^M u_{ik}^{\phi_U-1} \prod_{j=2}^N v_{kj}^{\phi_V-1}\right) |(\det J)_{k}| \\
&= C \iint du_{1k}dv_{k1} u_{1k}^{2z} v_{k1}^{2z}u_{1k}^{M\phi_U-M}v_{k1}^{N\phi_V-N} u_{1k}^{M-1} v_{k1}^{N-1}\\
&= C \iint du_{1k}dv_{k1} u_{1k}^{2z+M\phi_U-1} v_{k1}^{2z+N\phi_V-1} \\
&= D \frac{1}{2z + M\phi_U} \times \frac{1}{2z + N\phi_V},
\end{align*}
where $C$ and $D$ are positive constants.
The poles are $z=-M\phi_U/2, -N\phi_V/2$; therefore, the RLCT $\lambda_k$ is
$$\lambda_k=\frac{\min\{M\phi_U, N\phi_V\}}{2}.$$

Let $U_k$ be $(u_{ik})_{i=1}^M$, $V_k$ be $(v_{kj})_{j=1}^N$, and $\Phi_k(U_k,V_k)$ be $\Biggl(\sum_{i=1}^M u_{ik}^2 \Biggr) \Biggl( \sum_{j=1}^N v_{kj}^2 \Biggr)$ for each $k$.
The RLCT $\lambda$ becomes the sum over $k$;
$$\lambda = \sum_{k=1}^H \lambda_k,$$
since $\Phi(U,V) \simeq \sum_{k=1}^H \Phi_k(U_k,V_k)$ and the RLCT of $\Phi_k$ is $\lambda_k$.
Thus, we have
$$\lambda = \frac{H \min\{M\phi_U, N\phi_V\}}{2}.$$
\end{proof}

\section{Proof of Theorem~\ref{thm:main}}\label{pf-main}

\begin{proof}[Proof of Theorem\ref{thm:main}]

From the zeta function in Definition~\ref{def:NMFRLCT}, we consider the zero points of
the algebraic set $\{(U,V) \mid \| UV-U_0V_0 \|^2\varphi(U,V)=0\}$.
Then, we have
\begin{eqnarray*}
  && \| UV-U_0V_0 \|^2  \\
  &=& \sum_{i=1}^M \sum_{j=1}^N (u_{i1}v_{1j} + ... + u_{iH}v_{Hj} 
  - u^0_{i1}v^0_{1j} - ... - u^0_{iH_0}v^0_{H_0j})^2 \\
  &=& \sum_{i=1}^M \sum_{j=1}^N \Biggl( \sum_{k=1}^H u_{ik}v_{kj} - \sum_{k=1}^{H_0} u^0_{ik}v^0_{kj} \Biggr)^2 \\
  &=& \sum_{i=1}^M \sum_{j=1}^N \Biggl( \sum_{k=1}^{H_0} (u_{ik}v_{kj} - u^0_{ik}v^0_{kj}) 
+ \sum_{k=H_0 +1}^H u_{ik}v_{kj} \Biggr)^2 \\
  & \leqq & C \sum_{i=1}^M \sum_{j=1}^N 
\left( \sum_{k=1}^{H_0}  (u_{ik}v_{kj} - u^0_{ik}v^0_{kj})^2 + \sum_{k=H_0 +1}^H u_{ik}^2 v_{kj}^2 \right)  \ \rm{for} \ \exists C>0(\rm{const.}) \\
  & \sim & \sum_{i=1}^M \sum_{j=1}^N 
\left( \sum_{k=1}^{H_0}  (u_{ik}v_{kj} - u^0_{ik}v^0_{kj})^2 + \sum_{k=H_0 +1}^H u_{ik}^2 v_{kj}^2 \right) \\
  &=& \sum_{k=1}^{H_0} \Biggl( \sum_{i=1}^M \sum_{j=1}^N (u_{ik}v_{kj} - u^0_{ik}v^0_{kj})^2 \Biggr) 
+\sum_{k=H_0 +1}^H \Biggl( \sum_{i=1}^M \sum_{j=1}^Nu_{ik}^2 v_{kj}^2 \Biggr) \\
  &\sim& \sum_{k=1}^{H_0}
\left\|
 u_k (v_k)^T - u^0_k (v^0_k)^T
\right\|^2
+ \left\|
  \left(
    \begin{array}{ccc}
  u_{1(H_0 +1)} & \ldots & u_{1H} \\
  \vdots & \ddots & \vdots \\
  u_{M(H_0 +1)} & \ldots & u_{MH} \\
    \end{array}
  \right) 
  \left(
   \begin{array}{ccc}
  v_{(H_0 +1)1} & \ldots & v_{(H_0 +1)N} \\
  \vdots & \ddots & \vdots \\
  v_{H1} & \ldots & v_{HN} \\
    \end{array}
  \right)
\right\|^2.
\end{eqnarray*}

Write the first and second terms above as
$$K_1(U,V)=\sum_{k=1}^{H_0}\left\| u_k (v_k)^T - u^0_k (v^0_k)^T\right\|^2$$
and
$$K_2(U,V)= 
\left\|
  \left(
    \begin{array}{ccc}
  u_{1(H_0 +1)} & \ldots & u_{1H} \\
  \vdots & \ddots & \vdots \\
  u_{M(H_0 +1)} & \ldots & u_{MH} \\
    \end{array}
  \right) 
  \left(
   \begin{array}{ccc}
  v_{(H_0 +1)1} & \ldots & v_{(H_0 +1)N} \\
  \vdots & \ddots & \vdots \\
  v_{H1} & \ldots & v_{HN} \\
    \end{array}
  \right)
\right\|^2,
$$
respectively. Because the prior $\varphi(U,V)\geqq 0$, all we have to do is find the RLCT of
$$K_1(U,V) \varphi(U,V) + K_2(U,V) \varphi(U,V)$$
to derive an upper bound.

Let $\overline{\lambda}$ be the RLCT of the right side, $\lambda_1$ be an RLCT of first term in the right side, and $\lambda_2$ be the RLCT of the second one. If $\varphi(U,V)=0$, then $K_1(U,V) \ne 0$, and, thus, the first prior term can be ignored when calculating the RLCT; we have
$$K_1(U,V) + K_2(U,V) \varphi(U,V)$$
Because the variables are independent and RLCTs are order isomorphic, 
\begin{eqnarray}
  \label{rlctineq}
  \lambda \leqq \overline{\lambda} = \lambda_1 + \lambda_2.
\end{eqnarray}
Since $K_1(U,V)$ corresponds to a proof of Lemma \rm{\ref{lemH}} in the case that $H$ is replaced by $H_0$, we have
$$\lambda_1 = H_0 \frac{M+N-1}{2}.$$
In contrast, $K_2(U,V)$ corresponds to Lemma \rm{\ref{lemH0}} when $H$ is replaced with $H-H_0$. This leads to
$$\lambda_2 = \frac{(H-H_0)\min\{M\phi_U,N\phi_V \}}{2}.$$
Using the above two equalities for inequality(\ref{rlctineq}), we find that
\begin{eqnarray*}
\lambda &\leqq& \overline{\lambda} \\
&=& \lambda_1 + \lambda_2\\
&=& H_0 \frac{M+N-1}{2} + \frac{(H-H_0)\min\{M\phi_U,N\phi_V \}}{2}. \\
\end{eqnarray*}
Therefore, the RLCT $\lambda$ satisfies the following inequality:
$$\lambda \leqq \frac{1}{2} \left[ (H-H_0)\min\{M\phi_U,N\phi_V\} +H_0 (M+N-1) \right].$$
\end{proof}

\section{Proofs of Theorem~\ref{thm:bayes} and \ref{thm:BvsVB}}\label{pf-main2}

\begin{proof}[Proof of Theorem~\ref{thm:bayes} and \ref{thm:BvsVB}]

First, we prove Theorem~\ref{thm:bayes}.
Owing to the equality (\ref{WatanabeBayesF}) and (\ref{WatanabeBayesG}),
we have
\begin{gather*}
F_n = nS_n + \lambda \log n -(m-1)\log\log n+ O_p(1), \\
\mathbb{E}[G_n] =\frac{\lambda}{n}+o \left(\frac{1}{n} \right).
\end{gather*}
On account of Theorem~\ref{thm:main}, 
$$\lambda \leqq \frac{1}{2} \left[ (H-H_0)\min\{M\phi_U,N\phi_V \} +H_0 (M+N-1) \right].$$
holds. 
Theorem~\ref{thm:bayes} follows from the above three formulas and $m \geqq 1$:
\begin{gather*}
F_n \leqq  nS_n + \frac{1}{2} \left[ (H-H_0)\min\{M\phi_U,N\phi_V \} +H_0 (M+N-1) \right] \log n + O_p(1), \\
\mathbb{E}[G_n] \leqq  \frac{1}{2n} \left[ (H-H_0)\min\{M\phi_U,N\phi_V \} +H_0 (M+N-1) \right]
+o\left(\frac{1}{n}\right).
\end{gather*}

Next, we prove Theorem~\ref{thm:BvsVB}.
Let $\overline{\lambda}$ be the upper bound of $\lambda$ in Theorem~\ref{thm:main}:
$$\overline{\lambda}=\frac{1}{2} \left[ (H-H_0)\min\{M\phi_U,N\phi_V \} +H_0 (M+N-1) \right]$$
In the same way as the above, we have
$$F_n \leqq  nS_n + \overline{\lambda} \log n + O_p(1).$$
Also, because of Theorem~\ref{KohjimaVBNMF}, 
$$\overline{F}_n = nS_n + \lambda_{vb} \log n + O_p(1)$$
holds, and we have
$$\lambda_{vb} = \begin{cases}
(H-H_0)(M\phi_U+N\phi_V) + \frac{1}{2}H_0(M+N), & \mbox{if } M\phi_U + N\phi_V<\frac{M+N}{2} \\
\frac{1}{2}H(M+N), & \mbox{otherwise}.
\end{cases}$$
Thus, we compute their difference,
\begin{align*}
\overline{F}_n - F_n &= (\lambda_{vb}-\lambda)\log n + O_p(1) \\
&\geqq (\lambda_{vb} -\overline{\lambda}) \log n + O_p(1).
\end{align*}
When $M\phi_U + N\phi_V<\frac{M+N}{2}$, we have
\begin{align*}
\lambda_{vb} - \overline{\lambda} &= (H-H_0)(M\phi_U+N\phi_V) + \frac{1}{2}H_0(M+N)
- \frac{1}{2} \left[ (H-H_0)\min\{M\phi_U,N\phi_V \} +H_0 (M+N-1) \right] \\
&=(H-H_0)\left[M\phi_U+N\phi_V -\frac{1}{2}\min\{M\phi_U,N\phi_V \} \right]
+\frac{1}{2}H_0(M+N-M-N+1) \\
&=\frac{1}{2}[(H-H_0)(M\phi_U+N\phi_V + M\phi_U+N\phi_V -\min\{M\phi_U,N\phi_V \} )+H_0] \\
&=\frac{1}{2}[(H-H_0)(M\phi_U+N\phi_V+\max\{M\phi_U,N\phi_V\})+H_0].
\end{align*}
On the other hand, if $M\phi_U + N\phi_V \geqq \frac{M+N}{2}$,
then
\begin{align*}
\lambda_{vb} - \overline{\lambda} &= \frac{1}{2}H(M+N)
- \frac{1}{2} \left[ (H-H_0)\min\{M\phi_U,N\phi_V \} +H_0 (M+N-1) \right] \\
&=\frac{1}{2}H(M+N)
- \frac{1}{2} (H-H_0)\min\{M\phi_U,N\phi_V \} -\frac{1}{2}H_0 (M+N) + \frac{1}{2}H_0 \\
&=\frac{1}{2}(H-H_0)(M+N) - \frac{1}{2} (H-H_0)\min\{M\phi_U,N\phi_V \} + \frac{1}{2}H_0 \\
&=\frac{1}{2}[(H-H_0)(M+N-\min\{M\phi_U,N\phi_V\})+H_0].
\end{align*}
This completes the proof of Theorem~\ref{thm:BvsVB}.
\end{proof}

\section*{Acknowledgments}

This research was partially supported by NTT DATA Mathematical Systems Inc.
We thank Professor Sumio Watanabe of Tokyo Institute of Technology for advice.
The author would like to express his appreciation to the editor and the reviewers for pointing out ways to improve this paper.



\bibliography{../../bib/bibs-full}

\begin{thebibliography}{10}
\expandafter\ifx\csname url\endcsname\relax
  \def\url#1{\texttt{#1}}\fi
\expandafter\ifx\csname urlprefix\endcsname\relax\def\urlprefix{URL }\fi
\expandafter\ifx\csname href\endcsname\relax
  \def\href#1#2{#2} \def\path#1{#1}\fi

\bibitem{Paatero}
P.~Paatero, U.~Tapper, Positive matrix factorization: A non-negative factor
  model with optimal utilization of error estimates of data values,
  Environmetrics 5~(2) (1994) 111--126, doi:10.1002/env.3170050203.

\bibitem{Cemgil}
A.~T. Cemgil, Bayesian inference in non-negative matrix factorisation models,
  Computational Intelligence and Neuroscience 2009~(4) (2009) 17, article ID
  785152.

\bibitem{Xu}
W.~Xu, X.~Liu, Y.~Gong, Document clustering based on non-negative matrix
  factorization, in: Proceedings of the 26th annual international ACM SIGIR
  conference on Research and development in information retrieval. Association
  for Computing Machinery, 2003, pp. 267--273.

\bibitem{Lee}
D.~D. Lee, H.~S. Seung, Learning the parts of objects with nonnegative matrix
  factorization, Nature 401 (1999) 788--791.

\bibitem{Kim}
H.~Kim, H.~Park, Sparse non-negative matrix factorizations via alternating
  non-negativity-constrained least squares for microarray data analysis,
  Bioinformatics 23~(12) (2007) 1495--1502, doi:10.1093/bioinformatics/btm134.
  PMID 17483501.

\bibitem{Kohjima}
M.~Kohjima, T.~Matsubayashi, H.~Sawada, Probabilistic non-negative
  inconsistent-resolution matrices factorization, in: Proceeding of CIKM '15
  Proceedings of the 24th ACM International on Conference on Information and
  Knowledge Management, Vol.~1, 2015, pp. 1855--1858.

\bibitem{Bobadilla2018recommender}
J.~Bobadilla, R.~Bojorque, A.~H. Esteban, R.~Hurtado, Recommender systems
  clustering using bayesian non negative matrix factorization, IEEE Access 6
  (2018) 3549--3564.

\bibitem{AkaikeAIC}
H.~Akaike, Information theory and an extension of the maximum likelihood
  principle, in: Selected papers of hirotugu akaike, Springer, 1998, pp.
  199--213.

\bibitem{SWatanabeBookMath}
S.~Watanabe, Mathematical theory of Bayesian statistics, CRC Press, 2018.

\bibitem{Schwarz1978BIC}
G.~Schwarz, Estimating the dimension of a model, The annals of statistics 6~(2)
  (1978) 461--464.

\bibitem{Watanabe1}
S.~Watanabe, Algebraic analysis for non-regular learning machines, Advances in
  Neural Information Processing Systems 12 (2000) 356--362, denver, USA.

\bibitem{Watanabe2}
S.~Watanabe, Algebraic geometrical methods for hierarchical learning machines,
  Neural Networks 13~(4) (2001) 1049--1060.

\bibitem{Yamazaki1}
K.~Yamazaki, S.~Watanabe, Singularities in mixture models and upper bounds of
  stochastic complexity, Neural Networks 16~(7) (2003) 1029--1038.

\bibitem{Aoyagi1}
M.~Aoyagi, S.~Watanabe, Stochastic complexities of reduced rank regression in
  bayesian estimation, Neural Networks 18~(7) (2005) 924--933.

\bibitem{Rusakov2005asymptotic}
D.~Rusakov, D.~Geiger, Asymptotic model selection for naive bayesian networks,
  Journal of Machine Learning Research 6~(Jan) (2005) 1--35.

\bibitem{Yamazaki2}
K.~Yamazaki, S.~Watanabe, Algebraic geometry and stochastic complexity of
  hidden markov models, Neurocomputing 69 (2005) 62--84, issue 1-3.

\bibitem{Zwiernik2011asymptotic}
P.~Zwiernik, An asymptotic behaviour of the marginal likelihood for general
  markov models, Journal of Machine Learning Research 12~(Nov) (2011)
  3283--3310.

\bibitem{Drton2017forest}
M.~Drton, S.~Lin, L.~Weihs, P.~Zwiernik, et~al., Marginal likelihood and model
  selection for gaussian latent tree and forest models, Bernoulli 23~(2) (2017)
  1202--1232.

\bibitem{nhayashi2}
N.~Hayashi, S.~Watanabe,
  \href{http://dx.doi.org/10.1016/j.neucom.2017.04.068}{Upper bound of bayesian
  generalization error in non-negative matrix factorization}, Neurocomputing
  266C~(29 November) (2017) 21--28.
\newline\urlprefix\url{http://dx.doi.org/10.1016/j.neucom.2017.04.068}

\bibitem{nhayashi5}
N.~Hayashi, S.~Watanabe, Tighter upper bound of real log canonical threshold of
  non-negative matrix factorization and its application to bayesian inference,
  in: IEEE Symposium Series on Computational Intelligence (IEEE SSCI), 2017,
  pp. 718--725.

\bibitem{nhayashi7}
N.~Hayashi, S.~Watanabe, Asymptotic bayesian generalization error in latent
  dirichlet allocation and stochastic matrix factorization, SN Computer Science
  (2020) 1--36To appear.

\bibitem{Nagata2008asymptotic}
K.~Nagata, S.~Watanabe, Asymptotic behavior of exchange ratio in exchange monte
  carlo method, Neural Networks 21~(7) (2008) 980--988.

\bibitem{Drton}
M.~Drton, M.~Plummer, A bayesian information criterion for singular models,
  Journal of the Royal Statistical Society Series B 79 (2017) 323--380, with
  discussion.

\bibitem{WatanabeK2006stochastic}
K.~Watanabe, S.~Watanabe, Stochastic complexities of gaussian mixtures in
  variational bayesian approximation, Journal of Machine Learning Research
  7~(Apr) (2006) 625--644.

\bibitem{Nakajima2007variational}
S.~Nakajima, S.~Watanabe, Variational bayes solution of linear neural networks
  and its generalization performance, Neural Computation 19~(4) (2007)
  1112--1153.

\bibitem{HoshinoT2005hmmvb}
T.~Hoshino, K.~Watanabe, S.~Watanabe, Stochastic complexity of variational
  bayesian hidden markov models, in: International Joint Conference on Neural
  Networks, Vol.~2, 2005, pp. 1114--1119 vol. 2.
\newblock \href {https://doi.org/10.1109/IJCNN.2005.1556009}
  {\path{doi:10.1109/IJCNN.2005.1556009}}.

\bibitem{Kohjima2017phase}
M.~Kohjima, S.~Watanabe, Phase transition structure of variational bayesian
  nonnegative matrix factorization, in: International Conference on Artificial
  Neural Networks, Springer, 2017, pp. 146--154.

\bibitem{Yamazaki2013comparing}
K.~Yamazaki, D.~Kaji, Comparing two bayes methods based on the free energy
  functions in bernoulli mixtures, Neural Networks 44 (2013) 36--43.

\bibitem{WatanabeK2012alternative}
K.~Watanabe, An alternative view of variational bayes and asymptotic
  approximations of free energy, Machine learning 86~(2) (2012) 273--293.

\bibitem{Atiyah1970resolution}
M.~F. Atiyah, Resolution of singularities and division of distributions,
  Communications on pure and applied mathematics 23~(2) (1970) 145--150.

\bibitem{Cohen}
J.~E. Cohen, U.~G. Rothblum, Nonnegative ranks, decompositions, and
  factorizations of nonnegative matrices, Linear Algebra and Its Applications
  190 (1993) 149--168.

\bibitem{WatanabeAIC}
S.~Watanabe, Asymptotic equivalence of bayes cross validation and widely
  applicable information criterion in singular learning theory, Journal of
  Machine Learning Research 11~(Dec) (2010) 3571--3594.

\bibitem{Bezanson2017julia}
J.~Bezanson, A.~Edelman, S.~Karpinski, V.~B. Shah,
  \href{https://doi.org/10.1137/141000671}{Julia: A fresh approach to numerical
  computing}, SIAM review 59~(1) (2017) 65--98.
\newline\urlprefix\url{https://doi.org/10.1137/141000671}

\end{thebibliography}
\bibliographystyle{elsarticle-num}






\end{document}